\documentclass[12pt,a4paper]{amsart}

\usepackage{amsmath,amsthm,amsfonts,amssymb,array,amscd}
\usepackage{amsmath,geometry,amssymb,amsfonts,amsthm,graphicx,enumerate,latexsym,tabularx,amscd}
\usepackage{refcheck} %% Checks whether enumerated equations are referred to or not.
\norefnames
\nocitenames
\theoremstyle{plain}
\usepackage{amsthm}
 \newtheorem{thm}{Theorem}[section]
 \newtheorem{cor}[thm]{Corollary}%[section]
\theoremstyle{definition}
 \newtheorem{dfn}[thm]{Definition}%[section]
 \newtheorem{rem}[thm]{Remark}%[section] 
 \numberwithin{equation}{section}
\theoremstyle{definition}
\theoremstyle{remark}
 \numberwithin{equation}{section}

\renewcommand{\le}{\leqslant}\renewcommand{\leq}{\leqslant}
\renewcommand{\ge}{\geqslant}\renewcommand{\geq}{\geqslant}

\usepackage{tikz}
\usepackage[pdftex]{hyperref}
\usetikzlibrary{matrix,arrows}
\newcommand{\bbC}{\mathbb{C}}
\newcommand{\bbF}{\mathbb{F}}

\renewcommand{\and}{\quad \mbox{and} \quad}  %% "and" (text in equation)
\renewcommand{\le}{\leqslant}\renewcommand{\leq}{\leqslant}
\renewcommand{\ge}{\geqslant}\renewcommand{\geq}{\geqslant}

\title{Lamprecht-Tate Formula}


\author[Biswas]{\bfseries Sazzad Ali Biswas}

\address{
Einstein Institute of Mathematics\\ % \hfill (Received 00 00 2010)\\
Hebrew University of Jerusalem\\ %\hfill (Revised  00 00 2010)\\
Givat Ram, Jerusalem 91904, Israel}
\email{sazzad.biswas@mail.huji.ac.il, sazzad.jumath@gmail.com}

\begin{document}

%{\begin{flushleft}\baselineskip9pt\scriptsize
%PUBLICATIONS DE L'INSTITUT MATH\'EMATIQUE\newline
%Nouvelle s\'erie, tome 87(101) (2010), od--do \hfill DOI:
%\end{flushleft}}
\vspace{10mm}
\setcounter{page}{1}
\thispagestyle{empty}

\begin{abstract}
For multiplicative characters of a non-Archimedean local field, we have a formula for epsilon factors
due to John Tate. Before Tate, Erich Lamprecht
also gave a formula for local epsilon factors of linear characters. In \cite{JT1}, Tate generalizes the formula for 
epsilon factors. In this paper, we give a very short and neat proof of the Lamprecht-Tate formula.
We also show that the 
famous twisting formula of Deligne is a special case of the Lamprecht-Tate formula.
\end{abstract}

\maketitle
%\tableofcontents

\section{\textbf{Introduction}}

By Langlands we can associate a local epsilon factor (also known as local constant) with each finite dimensional
continuous complex representation $\rho$ of the absolute Galois group $G_F$ of a non-Archimedean local field $F$. 
But there is 
no explicit formula for the local epsilon of an arbitrary local Galois representation. 
In 1952 \cite{EL}, Erich Lamprecht gave a formula (cf. Lemma 8.1 on p. 60 of \cite{RL})
for the local epsilon factor of a linear character
$\chi:F^\times\to\bbC^\times$. Later, John Tate gave a more explicit formula (cf. \cite{JT1}, p. 94 and Proposition 1) 
for the local epsilon
factor of a linear character. We call this formula the {\bf Lamprecht-Tate} formula.
In this paper, we give a very short and neat proof of this Lamprecht-Tate formula (cf. Theorem \ref{Theorem 6.1.1}).
We also show that Tate's
formula (\ref{eqn 3.11}) and Lamprecht's formula (cf. Corollary \ref{Corollary 6.1.2}(1)) 
are special cases of the Lamprecht-Tate formula (\ref{eqn 6.0.9}). 

Furthermore, we also show that the famous Deligne's twisting formula (cf. Corollary \ref{Corollary 6.1.2}(2))
of local epsilon factors is a corollary
of Theorem \ref{Theorem 6.1.1}.

\section{\textbf{Notation and Preliminaries}}

Let $F$  be a non-Archimedean local field of characteristic zero
, i.e., a finite extension of the field $\mathbb{Q}_p$ (field of $p$-adic numbers),
where $p$ is a prime.
Let $O_F$ be the 
ring of integers of $F$ and let $P_F=\pi_F O_F$ be the unique maximal ideal in $O_F$ 
and $\pi_F$ a uniformizer, i.e., a generator of $P_F$.
We denote by $q_F$ the cardinality of the residue field $k_F=O_F/P_F$ of $F$.
%The cardinality of the residue field $\bbF_q=O_F/P_F$
%of $F$ is $q_F$, i.e., $|\bbF_q|=q_F$.
Let $U_F=O_F-P_F$ be the group of units in $O_F$.
Let $P_{F}^{i}=\{x\in F:\nu_F(x)\geq i\}$ and for $i\geq 0$ define $U_{F}^{i}=1+P_{F}^{i}$
(with the proviso that $U_{F}^{0}=U_F=O_{F}^{\times}$).

\begin{dfn}[\textbf{Different and Discriminant}] 
 Let $K/F$ be a finite separable extension of a non-Archimedean local field $F$. 
 We define the \textbf{inverse different (or codifferent)}
 $\mathcal{D}_{K/F}^{-1}$ of $K$ over $F$ to be $\pi_{K}^{-d_{K/F}}O_K$, where $d_{K/F}$ is the largest integer such that 
 \begin{center}
  $\mathrm{Tr}_{K/F}(\pi_{K}^{-d_{K/F}}O_K)\subseteq O_F$,
 \end{center}
 where $\rm{Tr}_{K/F}$ is the trace map from $K$ to $F$.
Then the \textbf{different} is defined by:
\begin{center}
 $\mathcal{D}_{K/F}=\pi_{K}^{d_{K/F}}O_K$
\end{center}
and the \textbf{discriminant} $D_{K/F}$ is 
\begin{center}
 $D_{K/F}=N_{K/F}(\pi_{K}^{d_{K/F}})O_F$.
\end{center}
\end{dfn}

\begin{dfn}[\textbf{Conductor of characters}]
 
The conductor of any nontrivial additive character $\psi$ of $F$  is the largest integer $n(\psi)$ such that 
$\psi$ is trivial
on $P_{F}^{-n(\psi)}$, but nontrivial on $P_{F}^{-n(\psi)-1}$. 
We also consider the conductor $a(\chi)$ of a 
 nontrivial character $\chi: F^\times\to \mathbb{C}^\times$, i.e., the smallest integer $m\geq 0$ such 
 that $\chi$ is trivial
 on $U_{F}^{m}$. We say $\chi$ is {\bf unramified} if the conductor of $\chi$ is {\bf zero} and otherwise {\bf ramified}.
  We also recall here that for two characters $\chi_1$ and $\chi_2$ of $F^\times$ we have 
 $a(\chi_1\chi_2)\leq\mathrm{max}(a(\chi_1),a(\chi_2))$ with equality if $a(\chi_1)\neq a(\chi_2)$. 
\end{dfn}

\subsection{\textbf{Classical Gauss sums}}

Let $k_q$ be a finite field. Let $p$ be the characteristic of $k_q$; then the prime  field 
contained in $k_q$ is $k_p$, the finite field of order $p$.
The structure of the {\bf canonical} additive character $\psi_q$ of $k_q$ is the same as the structure of the canonical
character
$\psi_F$ of $F$ (cf. \cite{JT1}, p. 92), namely the composite of the trace map with the canonical character of the base field, i.e., 
\begin{center}
 $\psi_q=\psi_p\circ \rm{Tr}_{k_q/k_p}$,
\end{center}
where 
\begin{center}
 $\psi_p(x):=e^{\frac{2\pi i x}{p}}$ \hspace{.3cm} for all $x\in k_p$.
\end{center}

\textbf{Gauss sums:} Let $\chi$ be a multiplicative and $\psi$ an additive character of $k_q$. 
Then the Gauss sum $G(\chi,\psi)$ is defined
by 
\begin{equation}
 G(\chi,\psi)=\sum_{x\in k_{q}^{\times}}\chi(x)\psi(x).
\end{equation}
In the following theorem we recall a known explicit formula for the Gauss sum associated with 
the nontrivial quadratic character
of $k_q^\times$.
\begin{thm}[\cite{LN}, p. 199, Theorem 5.15]\label{Theorem 3.5}
Let $k_q$ be a finite field with $q=p^s$ elements, where $p$ is an odd prime and $s\in\mathbb{N}$.
Let $\chi$ be the nontrivial quadratic character of 
$k_q$ and let $\psi$ be the canonical additive character of $k_q$. Then
\begin{equation}
 G(\chi,\psi)=\begin{cases}
               (-1)^{s-1}q^{\frac{1}{2}} & \text{if $p\equiv 1\pmod{4}$},\\
               (-1)^{s-1}i^sq^{\frac{1}{2}} & \text{if $p\equiv 3\pmod{4}$}.
              \end{cases}
\end{equation}
\end{thm}

\subsection{\textbf{Epsilon factors}}

For a nontrivial multiplicative character $\chi$ of $F^\times$ and a nontrivial additive character $\psi$ of $F$,
we have (cf. \cite{RL}, p. 5)
\begin{equation}\label{label1}
 \epsilon(\chi,\psi)=\chi(c)\frac{\int_{U_F}\chi^{-1}(x)\psi(x/c) dx}{|\int_{U_F}\chi^{-1}(x)\psi(x/c) dx|},
\end{equation}
where the Haar measure $dx$ of $F$ is normalized so that the normalized Haar measure of $O_F$ is $1$, and 
 $c\in F^\times$ is an element with valuation $n(\psi)+a(\chi)$.
Due to Tate (cf. p. 94 of \cite{JT1}), we can modify the above integral formula (\ref{label1}) as follows:
\begin{equation}\label{eqn 3.11}
 \epsilon(\chi,\psi)=\chi(c)q_F^{-\frac{a(\chi)}{2}}\sum_{x\in\frac{U_F}{U_{F}^{a(\chi)}}}\chi^{-1}(x)\psi(x/c),
\end{equation}
where $c=\pi_{F}^{a(\chi)+n(\psi)}$. We call the equation (\ref{eqn 3.11}) as {\bf Tate's formula.}

\section{{\bf Lamprecht-Tate formula for epsilon factors}}

\begin{thm}[{\bf Lamprecht-Tate formula, Proposition 1 of \cite{JT1}}]\label{Theorem 6.1.1}
Let $F$ be a non-Archimedean local field.
Let $\chi$ be a character of $F^\times$ of conductor $a(\chi)$ and let $m$ be a natural number 
such that $2m\le a(\chi)$. Let $\psi$ be a nontrivial additive character of $F$. 
Then there exists $c\in F^\times$, with valuation $\nu_F(c)=a(\chi)+n(\psi)$, such that 
\begin{equation}\label{eqn 5.4.5}
 \chi(1+y)=\psi(c^{-1}y)\qquad\text{for all $y\in P_{F}^{a(\chi)-m}$},
\end{equation}
and for such a $c$ we have:
\begin{equation}\label{eqn 6.0.9}
 \epsilon(\chi,\psi)=\chi(c)\cdot q_{F}^{-\frac{(a(\chi)-2m)}{2}}
 \sum_{x\in U_F^m/U_F^{a(\chi)-m}}\chi^{-1}(x)\psi(c^{-1}x).
\end{equation}
{\bf Remark:} Note that the assumption (\ref{eqn 5.4.5}) is obviously fulfilled for $m=0$ because then both sides are
equal to $1$,
and the resulting formula for $m=0$ is Tate's formula (\ref{eqn 3.11}).
\end{thm}

\begin{proof}
When $m=0$, the formula (\ref{eqn 6.0.9}) is the same as the Tate's formula (\ref{eqn 3.11}).
In general, the assumption $2m\le a(\chi)$ implies 
 $2(a(\chi)-m)\ge a(\chi)$ and therefore 
 $$\chi(1+y)\chi(1+y')=\chi(1+y+y')$$
 for $y,y'\in P_{F}^{a(\chi)-m}$. That is, $y\mapsto \chi(1+y)$ is a character of the additive group $P_F^{a(\chi)-m}$.
 This character extends to an additive character of the field $F$ and, by local additive duality, there is some $c\in F^\times$ such that
 $$\chi(1+y)=\psi(c^{-1}y)=(c^{-1}\psi)(y),\quad\text{for all $y\in P_F^{a(\chi)-m}$}.$$
 Now comparing the conductors of both sides we must have:
 $$a(\chi)=-n(c^{-1}\psi)=\nu_F(c)-n(\psi),$$
 hence $\nu_F(c)=a(\chi)+n(\psi)$ is the right assumption for our formula.\\
 Now we assume $m\ge 1$ (the case $m=0$ we have checked already) and consider the filtration 
 $$O_F^\times\supseteq 1+P_F^{a(\chi)-m}\supseteq 1+P_{F}^{a(\chi)}.$$
 Then we may represent every $x\in O_F^\times/(1+P_F^{a(\chi)})$ as $x=z(1+y)$, where $y\in P_F^{a(\chi)-m}$ and $z$ runs over the 
 system of representatives for $O_F^\times/(1+P_F^{a(\chi)-m})$. Now for giving explicit formula for the epsilon 
 factor $\epsilon(\chi,\psi)$, 
 we have to consider the sum 
 \begin{equation}\label{eqn 5.4.6}
  \sum_{x\in O_F^\times/(1+P_F^{a(\chi)})}\chi^{-1}(x)\psi(c^{-1}x)=\sum_{z\in O_F^\times/(1+P_F^{a(\chi)-m})}
  \sum_{y\in P_F^{a(\chi)-m}/P_{F}^{a(\chi)}}\chi^{-1}(z(1+y))\psi(c^{-1}z(1+y)).
 \end{equation}
Using (\ref{eqn 5.4.5}) we obtain
$$\chi^{-1}(z(1+y))=\chi^{-1}(z)\chi^{-1}(1+y)=\chi^{-1}(z)\chi(1-y)=\chi^{-1}(z)\psi(-c^{-1}y)$$
 and therefore our double sum (\ref{eqn 5.4.6}) may be rewritten as 
 \begin{equation*}
  \sum_{z\in O_F^\times/(1+P_F^{a(\chi)-m})}\chi^{-1}(z)\psi(c^{-1}z)\cdot
  \left(\sum_{y\in P_F^{a(\chi)-m}/P_F^{a(\chi)}}\psi(c^{-1}y(z-1))\right).
 \end{equation*}
But the inner sum is the sum over the additive group $P_F^{a(\chi)-m}/P_F^{a(\chi)}$ and 
$(c^{-1}(z-1))\psi$ is a character of that group. Hence this sum is equal to $[P_F^{a(\chi)-m}:P_F^{a(\chi)}]=q_{F}^{m}$
if the character is trivial and otherwise the sum will be zero. But:
$$n(c^{-1}(z-1)\psi)=\nu_F(c^{-1}(z-1))+n(\psi)=-a(\chi)+\nu_F(z-1).$$
So that the character $(c^{-1}(z-1))\psi$ is trivial on $P_F^{a(\chi)-\nu_F(z-1)}$, and therefore it will be $\equiv1$
on $P_{F}^{a(\chi)-m}$ if and only if $\nu_F(z-1)\ge m$, i.e., $z=1+y'\in 1+P_F^m$. Therefore our sum (\ref{eqn 5.4.6})
rewrites as 
  \begin{equation}\label{eqn 5.4.7}
  \sum_{x\in O_F^\times/(1+P_F^{a(\chi)})}\chi^{-1}(x)\psi(c^{-1}x)=
  q_F^m \sum_{z\in (1+P_F^m)/(1+P_F^{a(\chi)-m})}\chi^{-1}(z)\psi(c^{-1}z).
 \end{equation}
 And substituting this result into the Tate's formula (\ref{eqn 3.11}) we get 
 \begin{align*}
 \epsilon(\chi,\psi)
  &=\chi(c)q_F^{-\frac{a(\chi)}{2}}\sum_{x\in O_F^\times/(1+P_F^{a(\chi)})}\chi^{-1}(x)\psi(c^{-1}x)\\
  &=\chi(c)\cdot q_{F}^{-\frac{(a(\chi)-2m)}{2}}\sum_{z\in (1+P_F^m)/(1+P_F^{a(\chi)-m})}\chi^{-1}(z)\psi(c^{-1}z)\\
  &=\chi(c)\cdot q_{F}^{-\frac{(a(\chi)-2m)}{2}}\sum_{x\in (1+P_F^m)/(1+P_F^{a(\chi)-m})}\chi^{-1}(x)\psi(c^{-1}x). 
 \end{align*}

\end{proof}

\begin{cor}\label{Corollary 6.1.2}
\begin{enumerate}
\item{\bf Lamprecht Formulas (cf. Lemma 8.1, p.60 of \cite{RL})}
Let $\chi$ be a character of $F^\times$. Let $\psi$ be a nontrivial additive character of $F$.
\begin{enumerate}
 \item When $a(\chi)=2d\, (d\ge 1)$, we have 
$$\epsilon(\chi,\psi)=\chi(c)\psi(c^{-1}).$$
\item When $a(\chi)=2d+1\, (d\ge1)$, we have 
$$\epsilon(\chi,\psi)=\chi(c)\psi(c^{-1})\cdot q_F^{-\frac{1}{2}}\sum_{x\in P_F^d/P_F^{d+1}}\chi^{-1}(1+x)\psi(c^{-1}x).$$
\end{enumerate}
\item {\bf Deligne's twisting formula (cf. \cite{D1}, Lemma 4.16):}
If $\alpha$ and $\beta$ are characters of $F^\times$ that satisfy $a(\alpha)\ge 2\cdot a(\beta)$, then 
\begin{equation}\label{Deligne's twisting formula}
 \epsilon(\alpha\beta,\psi)=\beta(c)\cdot \epsilon(\alpha,\psi).
\end{equation}
\end{enumerate}
Here $c\in F^\times$ is an element with valuation $\nu_F(c)=a(\chi)+n(\psi)$,
and which satisfies the additional condition that 
\begin{center}
 $\chi(1+x)=\psi(\frac{x}{c})$ for all $x\in F^\times$ with $2\cdot\nu_F(x)\ge a(\chi)$
\end{center}
in case (1(a)) and case (1(b)), and 
$$\alpha(1+x)=\psi(x/c)\qquad \text{for all $x\in F^\times$ with $2\cdot\nu_F(x)\ge a(\alpha)$}$$
in case (2).
\end{cor}

\begin{proof}

We deduce the three assertions from formula (\ref{eqn 6.0.9}).\\
{\bf (1(a)).} Assume $a(\chi)=2d$, where $d\ge 1$. In this case, we take $m=d$, and from equation (\ref{eqn 6.0.9}) we obtain
 \begin{equation}
  \epsilon(\chi,\psi)=\chi(c)\cdot \sum_{x\in (1+P_F^d)/(1+P_F^d)}\chi^{-1}(x)\psi(c^{-1}x)=\chi(c)\cdot\psi(c^{-1}).
 \end{equation}
{\bf (1(b)).} Assume $a(\chi)=2d+1$, where $d\ge 1$. In this case, we also take $m=d$, and then from equation (\ref{eqn 6.0.9}) we obtain
\begin{align*}
 \epsilon(\chi,\psi)
 &=\chi(c)\cdot q_F^{-\frac{1}{2}}\cdot\sum_{x\in(1+P_F^d)/(1+P_F^{d+1})}\chi^{-1}(x)\psi(c^{-1}x)\\
 &=\chi(c)\cdot\psi(c^{-1})\cdot q_F^{-\frac{1}{2}}\cdot\sum_{x\in P_F^d/P_F^{d+1}}\chi^{-1}(1+x)\psi(c^{-1}x).
\end{align*}
Here $c\in F^\times$ with $\nu_F(c)=a(\chi)+n(\psi)$ satisfies 
\begin{center}
 $\chi(1+x)=\psi(c^{-1}x)$ for all $x\in F^\times$ with $2\nu_F(x)\ge a(\chi)$.
\end{center}
 {\bf (2) Proof of Deligne's twisting formula:}
By the given assumption  
$a(\alpha)\ge 2 a(\beta)$, we have $a(\alpha\beta)=a(\alpha)$. 
Now take $m=a(\beta)$, then from equation (\ref{eqn 6.0.9}) we can write:
\begin{align*}
 \epsilon(\alpha\beta,\psi)
 &=\alpha\beta(c)\cdot q_{F}^{-\frac{(a(\alpha)-2m)}{2}}\cdot\sum_{x\in (1+P_F^m)/(1+P_F^{a(\chi)-m})}(\alpha\beta)^{-1}(x)\psi(\frac{x}{c})\\
 &=\beta(c)\cdot \alpha(c)q_F^{-\frac{(a(\alpha)-2m)}{2}}\sum_{x\in(1+P_F^m)/(1+P_F^{a(\alpha)-m})}\alpha^{-1}(x)\psi(\frac{x}{c})\\
 &=\beta(c)\cdot \epsilon(\alpha,\psi),
\end{align*}
since, as $a(\beta)=m$, we have $\beta(x)=1$ for all $x\in (1+P_F^m)/(1+P_F^{a(\alpha)-m})$.

\end{proof}

\begin{rem}\label{Remark Appendix}

Let $\mu_{p^\infty}$ denote the group of roots of unity of $p$-power order. 
%Let $F$ be a non-Archimedean local field. Let $\psi_F$ be an additive character of $F$. Since $\psi_F$ is additive, its image lies 
%in $\mu_{p^{\infty}}$. We also know that $U_F^1$ is a pro-p-group, hence $\chi(U_F^1)\subset \mu_{p^{\infty}}$.

%\begin{enumerate}
% \item When $a(\chi)$ is even, we have
% \begin{equation}\label{eqn 5.2.8}
 % \epsilon(\chi,\psi_F)=\chi(c)\cdot\psi_F(c^{-1}).
 %\end{equation}
%\item When $a(\chi)=2d+1$ is odd, we have
%\begin{equation}\label{eqn 5.2.9}
% \epsilon(\chi,\psi_F)=\chi(c)\cdot\psi_F(c^{-1})\cdot q_{F}^{-\frac{1}{2}}\sum_{x\in P_{F}^{d}/P_{F}^{d+1}}\chi^{-1}(1+x)\psi_F(x/c).
%\end{equation}
%Here the constant $c=c_\chi\in F^\times$ is given by $\chi(1+x)=\psi(cx)$ if $2\nu_F(x)\ge a(\chi)$.
%\end{enumerate}
Now we can give an explicit formula for $\epsilon(\chi,\psi)$ modulo $\mu_{p^\infty}$ (cf. the Mathscinet review of 
Henniart's paper \cite{GH} by E.-W. Zink)
and which is:
%For this it is sufficient
%to know $c\in F^\times\rm{mod}\, 1+P_{F}$.
%
%As for the correcting term it is $\equiv 1\mod{\mu_{p^\infty}}$ in case $p=2$. If $p\ne 2$ we compare the function 
%$Q(x):=\chi^{-1}(1+x)\cdot\psi_F(x/c)$ and $H(x):=\psi_F(\frac{x^2}{2c})$ on $P_F^d$. 
%It is easy to see (see the review of Henniart's artile \cite{GH} by E.-W. Zink) 
%that $(1+x)(1+y)=(1+x+y)(1+\frac{xy}{1+x+y})$. Then we have 
%\begin{equation}
% \frac{Q(x+y)}{Q(x)\cdot Q(y)}=\chi(1+\frac{xy}{1+x+y})=\psi_F(\frac{xy}{c}).
%\end{equation}
%Similarly,
%\begin{equation}
% \frac{H(x+y)}{H(x)\cdot H(y)}=\psi_F(\frac{xy}{c}).
%\end{equation}
%Then 
%\begin{equation}\label{eqn 5.2.11}
%\frac{Q(x+y)}{Q(x)\cdot Q(y)}=\frac{H(x+y)}{H(x)\cdot H(y)}=\psi_F(\frac{xy}{c}).
%\end{equation}
%Therefore $Q$ and $H$ differ only by {\bf an additive} character on $P_{F}^{d}$ and then we can write:
\begin{enumerate}
 \item $\epsilon(\chi,\psi)\equiv\chi(c)\mod{\mu_{p^\infty}}$ if $a(\chi)$ is even,
 \item $\epsilon(\chi,\psi)\equiv\chi(c)G(c)\mod{\mu_{p^\infty}}$ if $a(\chi)=2d+1\, (d\ge 1)$, where 
 $$G(c):=q_{F}^{-\frac{1}{2}}\cdot\sum_{x\in P_F^d/P_{F}^{d+1}}\psi(\frac{x^2}{2c})$$
 depends only on $c\in F^\times\mod{1+P_F}$.
 
\end{enumerate}

\end{rem}

\begin{rem}[{\bf On $G(c)$}]
Let $\bbF_q$ be a finite field of odd cardinal $q$. Let $\psi_0$ be a nontrivial additive character of $\bbF_q$.
Let $\eta$ be the unique nontrivial quadratic character of $\bbF_q^\times$.
%Then from the generalization 
%(cf. \cite{BRK}, p. 47, Problem 24) of Theorem
%1.1.5 of \cite{BRK} on p. 12 we have 
Then we can write
\begin{equation}\label{eqn 6.1.11}
 \sum_{x\in \bbF_q}\psi_0(x^2)=\sum_{x\in\bbF_q^\times}\eta(x)\psi_0(x)=G(\eta,\psi_0).
\end{equation}
Now from Theorem \ref{Theorem 3.5} we can observe that $q^{-\frac{1}{2}}\sum_{x\in\bbF_q}\psi_0(x^2)$ is a {\bf fourth} root of unity.

Since $\nu_F(c)=a(\chi)+n(\psi)$, $P_F^d/P_F^{d+1}\cong k_{F}$ (here $d$ is same as in the previous remark),
we can write $G(c)$ as 
$$G(c)=q_F^{-\frac{1}{2}}\sum_{x\in k_F}\psi'(x^2)=q_F^{-\frac{1}{2}}G(\eta,\psi'),$$
where $\psi'$ is a certain nontrivial additive character of $k_F$ and $\eta$ is the unique nontrivial 
quadratic character of $k_F^\times$.
Therefore we can conclude that $G(c)$ is a fourth root of unity.

Furthermore, since $G(c)\cdot q_F^{\frac{1}{2}}=G(\eta,\psi')$ is a quadratic classical Gauss sum, from the 
computation of lambda function 
$\lambda_{K/F}(\psi):=\epsilon(Ind_{K/F}(1_K),\psi)$, where $K/F$ is a finite extension and 
$1_K$ is the trivial character of $K^\times$ 
 (via class field theory) (for details of lambda functions, see \cite{SAB1}). And when $K$ is a tamely quadratic ramified  
extension of $F$ we have explicit formula for $\lambda_{K/F}(\psi)$ ( see Theorem 1.1 of \cite{SAB2}) and which is a fourth 
root of unity. 
It can be seen that this $G(c)$ is actually
a $\lambda$-function for some quadratic tamely ramified local extension. Therefore explicit computation of $G(c)$ 
is equivalent to 
the computation of $\lambda$-factors for tamely quadratic ramified extensions $K/F$. 
 
\end{rem}

%\footnote{Since $\psi_F$ is
%additive character of $F$, $\psi_F(c^{-1})\equiv 1\mod{\mu_{p^\infty}}$, and we also know 
%that $U_{F}^{1}\subset F^\times$ is a pro-p-group, hence $\chi_\rho(U_{F}^{1})\subset \mu_{p^\infty}$.}

%\vspace{5mm}
%\textbf{Acknowledgements} 
%\thispagestyle{empty}
%I would like to thank Professor E.-W. Zink for encouraging me to work on local epsilon factors and his constant 
%valuable advices. 

\end{document}